\documentclass[12pt]{amsart}
\usepackage{amsmath}
\usepackage{amsfonts}
\usepackage{amssymb}
\usepackage{amscd}
\usepackage{graphicx}
\RequirePackage[dvipsnames,usenames]{color}
\usepackage{soul,xcolor}
\setstcolor{red}
\usepackage{stmaryrd}
\SetSymbolFont{stmry}{bold}{U}{stmry}{m}{n}
\usepackage{mathtools}
\usepackage{booktabs}
\usepackage{multirow}\usepackage{subcaption}

\usepackage{mathdots}
\newtagform{tiny}{\tiny(}{)}

\usepackage{mathtools}
\usepackage{hyperref}
\usepackage[margin=1.25in]{geometry}

\usepackage{amsthm}
\usepackage{comment}
\usepackage[all,cmtip]{xy}
\usepackage{tikz-cd}
\usetikzlibrary{cd}

\usepackage[all]{xy}

\usepackage{etoolbox}
\apptocmd{\sloppy}{\hbadness 10000\relax}{}{}

\usepackage{enumitem}

\title{Local vanishing for F-pure threefolds}
\author{Tatsuro Kawakami}
\address{Graduate School of Mathematical Sciences, University of Tokyo, 3-8-1 Komaba,
Meguro-ku, Tokyo 153-8914, Japan}
\email{tatsurokawakami0@gmail.com}

\def\phi{\varphi}
\def\epsilon{\varepsilon}

\def\mapsto{\longmapsto}

\def\log{\operatorname{log}}

\def\Spec{\operatorname{Spec}}

\def\Supp{\operatorname{Supp}}
\def\codim{\operatorname{codim}}

\def\Exc{\operatorname{Exc}}

\def\sg{\operatorname{sg}}

\def\max{\operatorname{max}}

\def\m{{\mathfrak m}}

\def\p{{\mathfrak p}}

\newcommand{\Q}{\mathbb{Q}}

\newcommand{\Z}{\mathbb{Z}}

\newcommand{\sO}{\mathcal{O}}
\newcommand{\sHom}{\mathcal{H}\! \mathit{om}}

\newsavebox{\pullback}
\sbox\pullback{%
\begin{tikzpicture}%
\draw (0,0) -- (1ex,0ex);%
\draw (1ex,0ex) -- (1ex,1ex);%
\end{tikzpicture}}

\newsavebox{\pullbackdl}
\sbox\pullbackdl{%
\begin{tikzpicture}%
\draw (-1ex,0ex) -- (0ex,0ex);%
\draw (0ex,-1ex) -- (0ex,0ex);%
\end{tikzpicture}}

\newsavebox{\pushoutdr}
\sbox\pushoutdr{%
\begin{tikzpicture}%
\draw (-1ex,-1ex) -- (-1ex,0ex);%
\draw (-1ex,0ex) -- (0ex,0ex);%
\end{tikzpicture}}

\theoremstyle{plain}
\newtheorem{thm}{Theorem}[section] 

\newtheorem{prop}[thm]{Proposition}

\newtheorem{lem}[thm]{Lemma}
\theoremstyle{definition} 
\newtheorem{defn}[thm]{Definition}

\theoremstyle{remark}
\newtheorem{rem}[thm]{Remark}

\newtheorem{defn and notation}[thm]{Definition and Notation}

\newtheorem*{cl}{Claim}

\theoremstyle{plain}
\newtheorem{theo}{Theorem}

\numberwithin{equation}{thm}

\begin{document}
\tolerance = 9999

\begin{abstract}
We establish Grauert--Riemenschneider vanishing for $F$-pure threefolds over a perfect field $k$ of characteristic $p>5$.
We apply this to prove Steenbrink vanishing for three-dimensional sharply $F$-pure pairs in characteristic $p>5$. As a consequence, we obtain the logarithmic extension for one-forms in this setting.
\end{abstract}

\maketitle
\markboth{TATSURO KAWAKAMI}{LOCAL VANISHING FOR F-PURE THREEHOLDS}

\section{Introduction}

In this article, our target is an $F$-pure threefold $X$ over a perfect field of characteristic $p>5$.
We first show Grauert--Riemenschneider vanishing for $X$.

\subsection{Grauert--Riemenschneider vanishing}
Grauert--Riemenschneider vanishing is one of essential tools in birational geometry. It can be seen as a local version of Kodaira vanishing.
Since there exist smooth projective surfaces in any positive characteristic that violates Kodaira vanishing, by taking an affine cone, we can see that Grauert--Riemenschneider vanishing fails in any positive characteristic.
On the other hand, it is famous that globally $F$-split varieties satisfies Kodaira vanishing, and therefore it is natural to expect Grauert--Riemenschneider vanishing for $F$-pure singularities, which is a local version of global $F$-splitting.
Indeed, three-dimensional strongly $F$-regular singularities, which are stronger version of $F$-pure singularities, satisfy Grauert--Riemenschneider vanishing \cite{Baudin-Kawakami-Rysler}.
However, very surprisingly, there exists an $F$-pure terminal singularity that violates Grauert--Riemenschneider vanishing \cite{Tot19,Baudin-Bernasconi-Kawakami}.
More strongly, this singularity is Frobenius liftable (see Proposition \ref{prop:Totaro}).

Nevertheless, we prove if $p>5$, then normal $F$-pure threefolds satisfy Grauert--Riemenschneider vanishing.

\begin{theo}[Grauert--Riemenschneider vanishing]\label{introthm:GR}
     Let $X$ be a normal $F$-pure threefold over a perfect field of characteristic $p>5$.
     Then we have
    \[
    R^j\pi_{*}\omega_Y=0
    \]
    for all $j>0$ and all resolutions $\pi\colon Y\to X$.
\end{theo}

Note that the condition $p\neq 2$ is essential for the vanishing as mentioned above. On the other hand, we do not know the condition that $p>5$ is optimal. 
We refer to Grauert--Riemenschneider vanishing for threefolds in positive characteristic for klt threefolds in characteristic $p>5$ \cite{Hacon-Witaszk(rationality),Bernasconi-Kollar},
strongly $F$-regular threefolds \cite{Baudin-Kawakami-Rysler}, and for lc threefolds in large characteristic \cite{Arvidsson}.

\subsection{Steenbrink vanishing vanishing}
In \cite{Kaw7}, we proved a partial version of Steenbrink vanishing for rational singularities, and in particular, we have seen that strongly $F$-regular threefolds and klt threefolds in characteristic $p>41$ satisfy this vanishing.
Combining Theorem \ref{introthm:GR} with the technique in \cite{Kaw7}, we can prove Steenbrink vanishing for three-dimensional sharply $F$-pure pairs in characteristic $p>5$.

\begin{theo}[Steenbrink vanishing]\label{introthm:Steenbrink}
     Let $(X,B)$ be a three-dimensional sharply $F$-pure pair over a perfect field of characteristic $p>5$ such that $K_X+B$ is $\Q$-Cartier.
     Let $\pi\colon Y\to X$ be a log resolution of $(X,B)$ whose reduced exceptional divisor $E$ supports a $\pi$-ample divisor.
     Set $B_Y\coloneqq  \pi^{-1}_*\lfloor B\rfloor+E$.
     Then we have
    \[
    R^j\pi_{*}\Omega^i_Y(\log B_Y)(-B_Y)=0
    \]
    for all integers $i,j$ satisfying $i+j>3$.
\end{theo}

Considering Steenbrink-type vanishing for lc varieties in characteristic zero in \cite[Theorem 14.1]{GKKP}, it is natural to expect the vanishing
\[
R^{2}\pi_{*}\Omega^i_Y(\log B_Y)(-B_Y)=0
    \]
    for all integers $i\geq 0$. The case $i \geq 2$ is covered by Theorem \ref{introthm:Steenbrink} and the case $i=0$ holds for varieties of lc type in any dimension in all characteristics (see Theorem \ref{thm:lc}).
    On the other hand, the case $i=1$ is widely open.

\subsection{Logarithmic extension for one-forms}

By combining Theorem \ref{introthm:Steenbrink} with the logarithmic extension for lc surface singularities with imperfect residue fields established in \cite{Kawakami-Sato}, we prove the logarithmic extension for three-dimensional sharply $F$-pure pairs in characteristic $p>5$.

\begin{theo}[Logarithmic extension for one-forms]\label{introthm:LET}
    Let $(X,B)$ be a three-dimensional sharply $F$-pure pair over a perfect field of characteristic $p>5$ such that $K_X+B$ is $\Q$-Cartier.
    Then $(X,B)$ satisfies the logarithmic extension theorem, i.e., the restriction map
    \[
    \pi_{*}\Omega^{1}_Y(\log B_Y)\hookrightarrow\Omega^{[1]}_X(\log \lfloor B\rfloor)
    \]
    is surjective for all log resolutions $\pi\colon Y\to X$ of $(X,B)$, where $B_Y\coloneqq  \pi^{-1}_*\lfloor B\rfloor+\Exc(\pi)$.
\end{theo}

Some logarithmic extension in higher dimension have been proven  in \cite{Kaw4,KTTWYY2,Kawakami-Sato}, and they all depends on a criterion using reflexive Cartier operator \cite[Theorem A]{Kaw4}.
However, it is not easy to apply \cite[Theorem A]{Kaw4} without Serre's condition $(S_3)$.
In Theorem \ref{introthm:LET}, by utilizing Grauert--Riemenschneider vanishing in Theorem \ref{introthm:GR} instead, we avoid $(S_3)$-condition.

\subsection{Lipman--Zariski type theorem}

In characteristic zero, it is conjectured that singularities with free tangent sheaves are smooth; this is known as the Lipman–Zariski conjecture \cite{Lipman65}. This conjecture is known to hold for log canonical singularities by \cite{Druel,Graf-Kovacs}.
In positive characteristic, however, the conjecture fails in general. Nevertheless, the freeness of the tangent sheaf is still expected to impose strong restrictions on singularities.
Indeed, in \cite[Theorem C]{Kawakami-Sato}, we show that strongly $F$-regular singularities with free tangent sheaves are Frobenius liftable, a highly restrictive condition (see also \cite[Theorem B]{Kawakami-Witaszek}).
Applying Theorem \ref{introthm:LET}, we obtain a similar result for $F$-pure threefolds:

\begin{theo}[Lipman--Zariski type theorem]\label{introthm:LZ-conjecture}
    Let $X$ be a normal $\Q$-Gorenstein $F$-pure threefold over a perfect field of characteristic $p>5$ with isolated singularities.
    If the tangent sheaf $T_X$ is locally free, then $X$ is Frobenius liftable.
\end{theo}

We will see in Proposition \ref{prop:Totaro} that there exists a terminal threefold $X$ in characteristic two such that $T_X$ is free, but $X$ is not smooth.

\section{Preliminaries}

\subsection{Notation and terminology}
Throughout the paper, we work over a fixed perfect field $k$ of characteristic $p > 0$ unless stated otherwise.
\begin{enumerate}
    \item A \textit{variety} is an integral separated scheme of finite type over $k$. 
    \item Given a proper birational morphism $f\colon Y\to X$ of Noetherian normal integral schemes, we denote $\Exc(f)$ by the reduced $f$-exceptional divisor.
    \item Given a Noetherian normal integral scheme $X$, a projective birational morphism $\pi \colon Y \to X$ is called \emph{a resolution of $X$} if $Y$ is regular.
    \item We say $(X,B)$ is a \emph{pair} if $X$ is a normal variety and $B$ is an effective $\Q$-divisor on $X$.
    \item Given a pair $(X,B)$, we say that $\pi\colon Y\to X$ is a \emph{log resolution} of $(X,B)$ if  $Y$ is smooth and $\mathrm{Supp}(\pi_{*}^{-1})B+\Exc(\pi)$ is snc.
    \item For the definition of the singularities of pairs appearing in the MMP (such as \emph{klt, dlt, lc}) we refer to \cite[Definition 2.8]{Kol13}. 
    \item Given a normal variety $X$ and a reduced divisor $B$ on $X$, we denote $j_{*}\Omega_U^{i}(\log B)$ by $\Omega_X^{[i]}(\log B)$, where $j\colon U\hookrightarrow X$ is the inclusion of the snc locus $U$ of $(X, B)$.
\end{enumerate}

\subsection{$F$-pure singularities}

In this section, we recall definition of an $F$-pure singularity and its variant.

\begin{defn}\label{def:F-pure}
Let $(X,B)$ be a pair of a normal variety $X$ and $B$ an effective $\Q$-divisor.
\begin{enumerate}
    \item We say $(X,B)$ is \textit{globally $F$-split} if the map
\[
\sO_{X}\to F^e_{*}\sO_{X}((p^e-1)B)
\] 
splits as an $\sO_{X}$-module homomorphism for all $e>0$.
\item We say $(X,B)$ is \textit{globally sharply $F$-split} if the map
\[
\sO_{X}\to F^e_{*}\sO_{X}(\lceil(p^e-1)B \rceil)
\] 
splits as an $\sO_{X}$-module homomorphism for some $e>0$.
\item We say that $(X,B)$ is \textit{$F$-pure} if there exists an affine open covering $\{U_i\}_{i\in I}$ such that $(U_i, B|_{U_i})$ is globally $F$-split.
\item We say that $(X,B)$ is \textit{sharply $F$-pure} if there exists an affine open covering $\{U_i\}_{i\in I}$ such that $(U_i, B|_{U_i})$ is globally sharply $F$-split.
\end{enumerate}
\end{defn}
\begin{rem}
We have the following implications: 
\[
\text{globally sharply $F$-split}\Longrightarrow
\text{globally $F$-split}, 
 \]
see \cite[Proposition 3.3]{Sch08g} for details. 
If every coefficient of $B$ is not divisible by $p$, then the other implication also holds.
\end{rem}

\subsection{Cartier operators}

Let $Y$ be a smooth variety, and let $E$ be a reduced divisor on $Y$ with snc support.
Let $A$ be a $\Q$-divisor whose support of the fractional part $A-\lfloor A\rfloor$ is contained in $E$.
Then we have locally free $\sO_X$-submodules
\[
B_{n}\Omega_Y^i(\log E)(p^nA)\subseteq Z_{n}\Omega_Y^i(\log E)(p^nA) \subseteq F^{n}_{*}\Omega_Y^i(\log E)(p^nA)
\]
of $F^{n}_{*}\Omega_Y^i(\log E)(p^nA)$. We refer to \cite[Section 3]{Kaw7} for details.
The iterated Cartier isomorphism (cf.\,\cite[Lemma 3.3]{Hara98}, \cite[equation (5.4.2)]{KTTWYY1}) gives the short exact sequence
\begin{equation}\label{BZOmega(iterated)}
    0 \to B_{n}\Omega_Y^i(\log E)(p^nA) \to Z_{n}\Omega_Y^i(\log E)(p^nA) \xrightarrow{C_{n}}  \Omega_Y^i(\log E)(A) \to 0
    \end{equation}
Taking $i=d$, we have a short exact sequence
\begin{equation}\label{BZOmega(iterated),i=d}
    0 \to B_{n}\Omega_Y^d(\log E)(p^nA) \to F^n_{*}\omega_Y(E+p^nA) \xrightarrow{C_{n}}  \omega_Y(E+A) \to 0,
    \end{equation}
and $C_n$ is nothing but the Frobenius trace map.

\subsubsection{Inverse Cartier operators}
For $n\geq 0$, we set
\begin{equation}\label{def of GOmega}
    G_n\Omega^i_Y(\log E)(p^nA) \coloneqq \frac{F^n_*\Omega^{i}_Y(\log E)(p^nA)}{B_n\Omega^{i}_Y(\log E)(p^nA)}.
\end{equation}
By definition, we have a short exact sequence
\begin{equation}\label{BOmegaG}
    0\to B_n\Omega^{i}_Y(\log E)(p^nA)\to F^n_*\Omega^{i}_Y(\log E)(p^nA)\to G_n\Omega^i_Y(\log E)(p^nA)\to 0.
\end{equation}
We have the inverse iterated Cartier operator
\begin{align}
\label{eq:dual-Cartier-Hara-higher} 
C^{-1}_{n}\colon \Omega^i_Y(\log E)(A)&\underset{\cong}{\xleftarrow{C_{n}}} \frac{Z_n\Omega^{i}_Y(\log E)(p^nA)}{B_n\Omega^{i}_Y(\log E)(p^nA)} \\[0.2em]
&\hookrightarrow \frac{F^n_*\Omega^{i}_Y(\log E)(p^nA)}{B_n\Omega^{i}_Y(\log E)(p^nA)}=G_n\Omega^i_Y(\log E)(p^nA). \nonumber
\end{align}
We have a decomposition 
\begin{equation}\label{eq:decomposition}
    C^{-1}_{n}\colon \Omega^i_Y(\log E)(A)\xrightarrow{C_{1,0}^{-1}} G_1\Omega^i_Y(\log E)(pA)\xrightarrow{C_{2,1}^{-1}}\cdots \xrightarrow{C_{n,n-1}^{-1}} G_n\Omega^i_Y(\log E)(p^nA)
\end{equation}
and a short exact sequence
\begin{equation}\label{GGB}
    0\to G_n\Omega^i_Y(\log E)(p^nA)\xrightarrow{C^{-1}_{n+1,n}}G_{n+1}\Omega^i_Y(\log E)(p^{n+1}A)\to F^n_{*}B\Omega^{i+1}_Y(\log E)(p^{n+1}A)\to 0.
\end{equation}
We refer to \cite[Section 3]{Kaw7} for details.

\subsubsection{Reflexive Cartier operator}
Let $X$ be a normal variety, and let $j \colon U\hookrightarrow X$ be the inclusion of the smooth locus.
We set 
\[
B_n \Omega_X^{[i]} \coloneqq j_{*}B_{n}\Omega_U^{i}\hspace{5mm}\text{and}\hspace{5mm} Z_n\Omega_X^{[i]} \coloneqq j_{*}Z_{n}\Omega_U^{i},
\]
which are reflexive $\sO_X$-modules.
Pushing forward the Cartier operator on $U$
\[
Z\Omega^i_U \xrightarrow{C}  \Omega^{i}_U, 
\]
we obtain the \textit{reflexive Cartier operator}
\[
Z\Omega^{[i]}_X \xrightarrow{C}  \Omega^{[i]}_X.
\]
Note that the reflexive Cartier operator is not necessarily surjective.
Indeed, the surjectivity of the reflexive Cartier operator of degree $i$ implies the logarithmic extension for $i$-forms \cite[Theorem A]{Kaw4}.

\section{Proof of main theorems}

In this section, we prove Theorems \ref{introthm:GR}, \ref{introthm:Steenbrink}, \ref{introthm:LET}, and \ref{introthm:LZ-conjecture}.

\subsection{Grauert--Riemenschneider vanishing}
We first focus on Theorem \ref{introthm:GR}. The key idea is to use the $F$-purity of a dlt modification of a sharply $F$-pure pair.

\begin{lem}\label{lem:key}
    Let $(X,B)$ be a sharply $F$-pure pair such that $K_X+B$ is $\Q$-Cartier, and let $\rho\colon Z\to X$ be a projective birational morphism with reduced exceptional divisor $E_Z$ such that
    \begin{enumerate}
        \item[\textup{(1)}] $K_Z+\rho^{-1}_*B+E_{Z}=\rho^{*}(K_X+B)$ and
        \item[\textup{(2)}] $E_Z$ supports a $\rho$-ample $\Q$-divisor.
    \end{enumerate}
    Then $R^j\rho_{*}\omega_Z=0$ for all $j>0$.
\end{lem}
\begin{proof}
    Since the assertion is local on $X$, we may assume that $(X,B)$ is globally sharply $F$-split.
    By assumption (1), it follows that $(Z,\rho^{-1}_*B+E_{Z})$ is globally sharply $F$-split \cite[Lemma 2.7]{GT}.
    In particular, $(Z,E_{Z})$ is globally $F$-split.
    By the definition of global $F$-splitting, the map
        \[
        \sO_Z\to F^e_{*}\sO_Z((p^e-1)E_Z)
        \]
        splits for all $e\geq 1$.
        Taking $\sHom_{\sO_Z}(-,\omega_Z)$, we have a split surjection
        \[
        F^e_{*}\omega_Z((1-p^e)E_{Z})\to \omega_Z
        \]
        for all $e\geq 1$.
         By assumption (2), we can take a $\rho$-ample $\Q$-divisor $A$ with $\lfloor A \rfloor=-E_{Z}$.
        We take $n\gg0$ such that 
        \[
        R^jf_{*}\omega_Z(E_{Z}+p^nA)=0
        \]
        for all $j>0$, whose existence is ensured by relative Serre vanishing.
        Now, we have the following commutative diagram:
        \[
\begin{tikzcd}
F^n_{*}\omega_Z((1-p^n)E_{Z}) \arrow[d,hookrightarrow] \arrow[r,"C_n"] & \omega_Z \arrow[d,equal] \\
F^n_{*}\omega_Z(E_{Z}+p^nA) \arrow[r,"C_n"] & \omega_Z(E_Z+A),
\end{tikzcd}  
\]
see \eqref{BZOmega(iterated),i=d}.
        Since the upper horizontal map is a split surjection, we have a surjection
        \[
        0=R^j\rho_{*}\omega_Z(E_{Z}+p^nA)\to R^j\rho_{*}\omega_Z
        \]
        for all $j>0$,
        which concludes the assertion.
\end{proof}

\begin{proof}[Proof of Theorem \ref{introthm:GR}]
   We first note that the vanishing can be checked one resolution since we have resolution of singularities in dimension three and smooth threefolds satisfy Grauert--Riemenschneider vanishing \cite{CR11,CR15}.
   By \cite[Theorem 4.3]{SS10}, we can find an effective $\Q$-divisor $B$ such that $(X,B)$ is sharply $F$-split.
   By using \cite[Theorem 1]{Kollar-Witaszek}, we can take a log resolution $\pi\colon Y\to X$ of $(X,B)$ with reduced exceptional divisor $E$ supports a $\pi$-ample divisor $H$.
    We fix such a resolution, and aim to show that 
    \[
    R^j\pi_{*}\omega_Y=0
    \]
    for all $j>0$.
    
    Now, we take a dlt modification as a factorization of $\pi$.
    Since $E$ supports a $\pi$-ample divisor, we can run a $(K_Y+\pi^{-1}_*B+E)$-MMP over $X$ scaling a $\pi$-ample divisor by \cite[Theorem 9]{Kollar(non-Q-Factorial)}.
    Thus, we obtain a factorization 
    \[
    \pi\colon Y\xrightarrow{\phi} Z\xrightarrow{\rho} X,
    \]
    where $\rho\colon Z\to X$ is a minimal model, which is called a \textit{dlt modification}.
    Let $E_Z$ be the reduced $\rho$-exceptional divisor.
    Then $E_Z$ supports a $\rho$-ample divisor \cite[Theorem 9 (3)]{Kollar(non-Q-Factorial)}.
    Since $(X,B)$ is sharply $F$-split and $K_X+B$ is $\Q$-Cartier, it is log canonical \cite[Theorem 4.4]{SS10}.
    Since $K_Z+\rho^{-1}_{*}B+E_{Z}$ is nef, the negativity lemma shows that $K_Z+\rho^{-1}_{*}B+E_{Z}=\rho^{*}(K_X+B)$.
    
    Now, we can use Lemma \ref{lem:key} to conclude $R^j\rho_{*}\omega_Z=0$ for all $j>0$.
    Furthermore, since $Z$ is klt and $p>5$, we have $R\phi_{*}\omega_Y=\omega_Z$ \cite[Theorem 1.3]{Bernasconi-Kollar}. We remark that we use the assumption that $p>5$ only here. 
    Now, we obtain $R^j\pi_{*}\omega_Y=R^j\phi_{*}R\rho_{*}\omega_Y=R^j\phi_{*}\omega_Z=0$ for all $j>0$, as desired.
\end{proof}

\subsection{Steenbrink vanishing}

In this section, we derive Steenbrink vanishing from the Grauert–Riemenschneider vanishing established in the previous section. The argument makes essential use of techniques involving the Cartier operator.
The key strategy of the proof (Proposition \ref{prop:Steenbrink}) is essentially the same as in \cite{Kaw7}.
However, since we cannot assume that the singularities are rational, several additional arguments are required.

\begin{prop}[\textup{cf.~\cite[Theorem 4.1 and Proposition 4.2]{Kaw7}}]\label{prop:Steenbrink}
    Let $X$ be a normal variety of dimension $d$, and let $B$ be a reduced divisor on $X$.
    Let $\pi\colon Y\to X$ be a log resolution of $(X,B)$ whose reduced exceptional divisor $E$ supports a $\pi$-ample divisor. Set $B_Y\coloneqq \pi^{-1}_{*}B+E$.
    Suppose that 
    \begin{itemize}
        \item $C\colon F_{*}R^{d-3}\pi_{*}\omega_Y\to R^{d-3}\pi_{*}\omega_Y$ is surjective and
        \item $C\colon F_{*}R^{d-2}\pi_{*}\omega_Y\to R^{d-2}\pi_{*}\omega_Y$ is injective.
    \end{itemize}
     Then 
    \[
    R^{d-1}\pi_{*}\Omega^{d-1}_Y(\log B_Y)(-B_Y)=0.
    \]
\end{prop}
\begin{proof}
    By the short exact sequence \eqref{BZOmega(iterated),i=d}
    \[
    0\to B\Omega^d_Y \to F_{*}\omega_Y\to \omega_Y\to 0
    \]
    and assumptions, we have 
    \[
    R^{d-2}\pi_{*}B\Omega^d_Y=0.
    \]
    Since $B\Omega^d_Y\cong B\Omega^d_Y(\log B_Y)(-B_Y)$ by \cite[Lemma 3.3 (3)]{Kaw7}, we have 
    \[
    R^{d-2}\pi_{*}B\Omega^d_Y(\log B_Y)(-B_Y)=0.
    \]
    By the short exact sequence \eqref{GGB}
    \begin{multline*}
        0\to G_{n}\Omega^{d-1}_Y(\log B_Y)(-B_Y) \xrightarrow{C_{n+1,n}^{-1}} G_{n+1}\Omega^{d-1}_Y(\log B_Y)(-B_Y)\\ \to F^nB\Omega^{d}_Y(\log B_Y)(-B_Y)\to 0,
    \end{multline*}
    it follows that
    \[
    C_{n+1,n}^{-1}\colon R^{d-1}\pi_{*}G_{n}\Omega^{d-1}_Y(\log B_Y)(-B_Y)\to R^{d-1}\pi_{*}G_{n+1}\Omega^{d-1}_Y(\log B_Y)(-B_Y)
    \]
    is injective for all $n\geq 0$, and thus
    \[
    C_{n}^{-1}\colon R^{d-1}\pi_{*}\Omega^{d-1}_Y(\log B_Y)(-B_Y)\to R^{d-1}\pi_{*}G_{n}\Omega^{d-1}_Y(\log B_Y)(-B_Y)
    \]
    is injective for all $n\geq 0$ by \eqref{eq:decomposition}.
    By assumption, we can take a $\pi$-ample anti-effective $\Q$-divisor $A$ such that $\lceil A \rceil=0$ and $\Supp(A)=E$.
   Take $n\gg0$ so that $A\leq -\frac{1}{p^n}E$ and $R^{d-1}\pi_{*}\Omega^{i}_Y(\log E)(-\pi^{-1}_{*}B+p^nA)=0$, whose existence is ensured by Serre vanishing.
   Now, we have the following commutative diagram:
    \[
\begin{tikzcd}
R^{d-1}\pi_{*}\Omega_Y^i(\log B_Y)(-\pi^{-1}_{*}B+A) \arrow[d,equal] \arrow[r,"C^{-1}_n"] & R^{d-1}\pi_{*}G_{n}\Omega_Y^i(\log E)(-\pi^{-1}_{*}B+p^nA) \arrow[d] \\
R^{d-1}\pi_{*}\Omega_Y^i(\log B_Y)(-B_Y) \arrow[r,"C^{-1}_n"] & R^{d-1}\pi_{*}G_{n}\Omega_Y^i(\log B_Y)(-B_Y).
\end{tikzcd}  
\]
Since the lower horizontal map is injective, the assertion is reduced to the vanishing
\[
R^{d-1}\pi_{*}G_{n}\Omega_Y^i(\log E)(-\pi^{-1}_{*}B+p^nA)=0.
\]
Consider the exact sequence \ref{BOmegaG}
\begin{multline*}
    0\to B_{n}\Omega_Y^i(\log E)(-\pi^{-1}_{*}B+p^nA)\to F^n_{*}\Omega_Y^i(\log E)(-\pi^{-1}_{*}B+p^nA)\\ \to G_{n}\Omega_Y^i(\log E)(-\pi^{-1}_{*}B+p^nA)\to 0.
\end{multline*}
By the assumption of $n$, we have
\[
R^{d-1}\pi_{*}F^n_{*}\Omega_Y^i(\log E)(-\pi^{-1}_{*}B+p^nA)=0.
\]
Since every fiber of $\pi$ has dimension at most $d-1$, we have
 \[
R^{d}\pi_{*}B_n\Omega_Y^i(\log E)(-\pi^{-1}_{*}B+p^nA)=0.
\]
Thus, we obtain the desired vanishing.
\end{proof}

\begin{proof}[Proof of Theorem \ref{introthm:Steenbrink}]
        We use the notation of Proof of Theorem \ref{introthm:GR}.
        By Theorem \ref{introthm:GR}, we have $R^j\pi_{*}\omega_Y=0$ for all $j>0$.
        Thus, it suffices to show that
        \[
        R^2\pi_{*}\Omega^2_Y(\log B_Y)(-B_Y)=0.
        \]
        Furthermore, by Proposition \ref{prop:Steenbrink}, this vanishing is reduced to the surjectivity of
        \[
        C\colon \pi_{*}F_{*}\omega_Y\to \pi_{*}\omega_Y.
        \]
        Consider a factorization 
    \[
    \pi\colon Y\xrightarrow{\phi} Z\xrightarrow{\rho} X,
    \]
    where $\rho\colon Z\to X$ is a dlt modification of $(X,B)$.
    Then as in the proof of Theorem \ref{introthm:GR},
    $Z$ is klt and globally $F$-split.
    Since $\phi_{*}\omega_Y=\omega_Z$, our aim is to show
    \begin{equation}\label{eq:long}
            C\colon \rho_*F_* \omega_Z\to \rho_*\omega_Z
        \end{equation}
        is surjective.
        Since $Z$ is globally $F$-split, the map
        \[
        C\colon F_* \omega_Z\to\omega_Z
        \]
        is a split surjection, and thus \eqref{eq:long} is surjective.
        Thus, we conclude.
\end{proof}

\subsection{Extension one-forms and Lipman-Zariski type theorem}

In this section, we derive logarithmic extension for one-forms on $F$-pure threefolds from the Steenbrink vanishing in the previous section and the logarithmic extension for two-dimensional lc singularities with imperfect residue fields \cite[Theorem 5.9]{Kawakami-Sato}.

\begin{proof}[Proof of Theorem \ref{introthm:LET}]
Since the assertion is independent from a choice of log resolutions (cf.~\cite[Lemma 2.13]{GKK}), we may assume that $E$ supports a $\pi$-ample divisor \cite[Theorem 1.1]{Kollar-Witaszek}.
We may assume that $X=\Spec\,R$ is affine.
It suffices to show that $\pi_{*}\Omega^{1}_Y(\log E)$ satisfies Serre's condition $(S_2)$, i.e., 
\[
H^1_{\p\sO_{X,\p}}\left((\pi_{*}\Omega^{1}_Y(\log E))_{\p}\right)=0
\]
for every prime ideal $\p\subseteq R$ whose hight bigger than one.
By \cite[Theorem 5.9]{Kawakami-Sato}, the restriction map
\[
\pi_{*}\Omega^{[1]}_Y(\log B_Y)\hookrightarrow \Omega^{[1]}_X(\log B)
\]
is an isomorphism in codimension two. Here, we note that $(X,\lfloor B\rfloor )$ is log canonical in codimension two \cite[Proposition 2.2 (2)]{BBKW}.
Therefore, it suffices to show that
\[
H^1_{\m_x}(\pi_{*}\Omega^{[1]}_Y(\log B_Y))=0
\]
for every closed point $x\in X$.
Consider the exact triangle
\[
\pi_{*}\Omega^{1}_Y(\log B_Y)\to R\pi_{*}\Omega^{1}_Y(\log B_Y)\to R^{>0}\pi_{*}\Omega^{1}_Y(\log B_Y) \xrightarrow{+1},
\]
where $R^{>0}\pi_{*}\Omega^{1}_Y(\log B_Y)$ denotes the canonical truncation $\tau^{>0}R\pi_{*}\Omega^{1}_Y(\log B_Y)$.
Since $R^{>0}\pi_{*}\Omega^{1}_Y(\log B_Y)\in D^{>0}$, we have
$H^0_{\m_x}(R^{>0}\pi_{*}\Omega^{1}_Y(\log B_Y))=0$.
Thus it suffices to show that
\[
H^1_{\m_x}(R\pi_{*}\Omega^{1}_Y(\log B_Y))=0.
\]
By local duality \cite[\href{https://stacks.math.columbia.edu/tag/0AAK}{Tag 0AAK}]{stacks-project}, we have isomorphisms
\begin{align*}
    (H^1_{\m_x}(R\pi_{*}\Omega^{1}_Y(\log B_Y)))^{\vee}&\cong \mathcal{H}^{-1}R\sHom_{\sO_Y}(R\pi_{*}\Omega^{1}_Y(\log B_Y),\omega^{\bullet}_Y)^{\wedge}_x\\
    &\cong \mathcal{H}^{-1}\left(R\pi_{*}\sHom_{\sO_Y}(\Omega^{1}_Y(\log B_Y),\omega_Y)[3]\right)^{\wedge}_x\\
    &\cong R^{2}\pi_{*}\Omega^{2}_Y(\log B_Y)(-B_Y)_x^{\wedge}.
\end{align*}
By Theorem \ref{introthm:Steenbrink}, we conclude.
\end{proof}

\subsection{Lipman-Zariski type theorem}
In this section, we prove a Lipman–Zariski type theorem as a consequence of the logarithmic extension theorem (Theorem \ref{introthm:LET}).
To this end, we establish the surjectivity of the reflexive Cartier operator using the logarithmic extension theorem (Proposition \ref{prop:surj of ref Cartier operator}).
It is known that the surjectivity of the reflexive Cartier operator in degree $i$ implies the logarithmic extension theorem for $i$-forms \cite[Theorem A]{Kaw4}.
The result below may be viewed as a converse to this implication.

\begin{prop}\label{prop:surj of ref Cartier operator}
    Let $X$ be a normal $F$-pure $\Q$-Gorenstein threefold over a perfect field of characteristic $p>5$ with isolated singularities.
    Then the reflexive Cartier operator
    \[
    C\colon Z\Omega^{[1]}_X\to \Omega^{[1]}_X
    \]
    is surjective.
\end{prop}
\begin{rem}
    If $X$ is Cohen--Macaulay, then the statement holds without any assumption on the characteristic $p$ (see \cite[Lemma 3.10]{Kaw4}).
    In this theorem we impose conditions on the dimension and the characteristic in the place of Serre's conditions.

    The main reason for assuming that $\dim X=3$ and $p>5$ is that we use the fact that three-dimensional klt singularities in characteristic $p>5$ are rational. 
\end{rem}
\begin{proof}
    Since the assertion is local on $X$, we may assume that $X$ is globally $F$-split.
    Let $\pi\colon Y\to X$ be a log resolution.
    As in the proof of Theorem \ref{introthm:GR}, we run a $(K_Y+E)$-MMP over $X$ \cite[Theorem 9]{Kollar(non-Q-Factorial)}
    to obtain the factorialization
    \[
    \pi\colon Y\xrightarrow{\phi} Z\xrightarrow{\rho} X,
    \]
    where $\rho\colon Z\to X$ is a dlt modification. 
    As in the proof of Theorem \ref{introthm:GR}, $Z$ is klt and globally $F$-split.
    By the definition of global $F$-splitting, the short exact sequence
    \[
    0\to \sO_Z\to F_{*}\sO_Z\to B\Omega^{[1]}_Z\to 0
    \]
    splits.
    Here, $F_{*}\sO_Z/\sO_Z$ is $S_2$ since it is a direct summand of $F_{*}\sO_Z$, which shows 
    an isomorphism $F_{*}\sO_Z/\sO_Z\cong B\Omega^{[1]}_Z$.
    Taking $R\rho_{*}$, we have a split exact sequence
    \[
    0\to R^j\rho_{*}\sO_Z\to F_{*}R^j\rho_{*}\sO_Z\to R^j\rho_{*}B\Omega^{[1]}_Z\to 0.
    \]
    Since regular excellent scheme has rational singularities \cite[Theorem 1.1]{CR15},
    the higher direct image $R^j\rho_{*}\sO_Z$ is supported at the singular locus, which has dimension zero.
    Thus, $R^j\rho_{*}\sO_Z$ is a finite-dimensional $k$-vector space, and $R^j\rho_{*}\sO_Z\to F_{*}R^j\rho_{*}\sO_Z$ is an isomorphism for all $j>0$. 
    Therefore, we obtain \begin{equation}\label{equation 1}
        R\rho_{*}B\Omega^{[1]}_Z=B\Omega^{[1]}_X.
    \end{equation}

    Next, consider the short exact sequence
    \[
    0\to \sO_Y\to F_{*}\sO_Y\to B\Omega^{1}_Y\to 0.
    \]
    Since $Z$ is klt and $p>5$, we have $R\phi_{*}\sO_Y=\sO_Z$ by \cite[Corollary 1.3]{ABL}.
    Taking $R\phi_{*}$, we have
    \begin{equation}\label{equation 2}
    R\phi_{*}B\Omega^1_Y=B\Omega^{[1]}_Z.
    \end{equation}
    Combining \eqref{equation 1} and \eqref{equation 2}, we have
    \begin{equation}\label{equation 3}
    R\pi_{*}B\Omega^1_Y=B\Omega^{[1]}_X.
    \end{equation}
    
    Consider the short exact sequence \eqref{BZOmega(iterated)}
    \[
    0\to B\Omega^1_Y\to Z\Omega^1_Y(\log E)\xrightarrow{C} \Omega^1_Y(\log E)\to 0,
    \]
    where we used $B\Omega^1_Y=B\Omega^1_Y(\log E)$ \cite[Lemma 3.3 (1)]{Kaw7}.
    Taking $\pi_{*}$, we have a surjection
    \[
    \pi_{*}Z\Omega^{1}_Y(\log E)\xrightarrow{C}\pi_{*}\Omega^{1}_Y(\log E)\to R^1\pi_{*}B\Omega^1_Y\overset{\eqref{equation 3}}{=}0.
    \]
    By Theorem \ref{introthm:LET}, we have 
    \[
    \pi_{*}\Omega^{1}_Y(\log E)\cong \Omega^{[1]}_X.
    \]
    By the exact sequence
    \[
    0\to \pi_{*}Z\Omega^{1}_Y(\log E) \to F_{*}\pi_{*}\Omega^{1}_Y(\log E)\cong F_{*}\Omega^{[1]}_X \to \pi_{*}B\Omega^{2}_Y(\log E),
    \]
    it follows that $\pi_{*}Z\Omega^{1}_Y(\log E)\cong Z\Omega^{[1]}_X$ since $\pi_{*}B\Omega^{2}_Y(\log E)$ is torsion-free.
    Thus we obtain the surjectivity of 
    \[
    C\colon Z\Omega^{[1]}_X\to \Omega^{[1]}_X,
    \]
    as desired.
\end{proof}

\begin{proof}[Proof of Theorem \ref{introthm:LZ-conjecture}]
    Since $\Omega^{[1]}_X\cong \sHom_{\sO_X}(T_X,\sO_X)$ is locally free, the surjection
    \[
    Z\Omega^{[1]}_X\xrightarrow{C} \Omega^{[1]}_X
    \]
    splits locally. This shows that $X$ is Frobenius liftable by \cite[Theorem 3.3]{Kawakami-Takamatsu}.
\end{proof}

\subsection{Totaro's example}

In the next proposition, we show that Totaro's example \cite{Tot19} of a non-Cohen--Macaulay terminal threefold is Frobenius liftable, but does not satisfy Grauert--Riemenschneider vanishing.
The failure of Grauert--Riemenschneider vanishing was proven in \cite[Theorem 5.2]{Baudin-Bernasconi-Kawakami}. In fact, stronger result is proved there: $R^1\pi_{*}\omega_Y$ is not annihilated by the action of the Cartier operator.
Following the approach of \cite[Remark 5.3]{Baudin-Bernasconi-Kawakami}, we show that this example is Frobenius liftable. Here, we say that a variety $X$ is \textit{Frobenius liftable} if it admits a flat lifting to the ring $W_2(k)$ of Witt vectors of length two together with a lifting of its Frobenius morphism (see \cite[Definition 2.6]{Kaw4}).
Since Frobenius liftable varieties satisfy the logarithmic extension of any degree \cite[Theorem B]{Kaw4},
we may still expect that Theorem \ref{introthm:LET} holds in any characteristic.

\begin{prop}[\textup{cf.\,\cite{Tot19,Baudin-Bernasconi-Kawakami}}]\label{prop:Totaro}
    Let $\mathbb{G}^3_{m}=\Spec\,\mathbb{F}_2[x,\frac{1}{x},y,\frac{1}{y},z,\frac{1}{z}]$ and consider a $\Z/2\Z$-action $\sigma\colon \mathbb{G}^3_{m}\to \mathbb{G}^3_{m}$ defined by 
    \[
    x\mapsto \frac{1}{x},y\mapsto\frac{1}{y},z\mapsto\frac{1}{z}.
    \]
    Let $\rho \colon \mathbb{G}^3_{m}\to X\coloneqq \mathbb{G}^3_{m}/\sigma$ be the quotient by the $\Z/2\Z$-action, and let $\pi\colon Y\to X$ be a log resolution with reduced exceptional divisor $E$.
    Then the following hold:
    \begin{enumerate}
        \item[\textup{(1)}] $X$ is terminal.
        \item[\textup{(2)}] $R^1\pi_{*}\omega_Y\neq 0$.
        \item[\textup{(3)}] $\Omega^{[i]}_X$ is free $\sO_X$-module for all $i\geq 0$.
        \item[\textup{(4)}] $C\colon Z\Omega^{[i]}_X\to \Omega^{[i]}_X$ is surjective for all $i\geq 0$.
        \item[\textup{(5)}] $\pi_{*}\Omega^i_Y(\log E)=\Omega^{[i]}_X$ for all $i\geq0$.
        \item[\textup{(6)}] $X$ is Frobenius liftable.
    \end{enumerate}
\end{prop}
\begin{rem}
    Suppose that $X$ is a terminal Cohen--Macaulay threefold such that $\Omega^{[1]}_X\cong \sO_X^{\oplus 3}$. Then $X$ is smooth as follows. Since $\omega_X\cong \sO_X$, the variety $X$ is Gorenstein and terminal. In particular, $X$ has only hypersurface singularities by \cite[Theorem 6.4]{TanakaGorensteinterminal}.
    Since $\codim_{X}(X_{\sg})\geq 3$, it follows that \[\Omega^1_X\cong \Omega^{[1]}_X\cong \sO_X^{\oplus 3}\] by \cite[Lemma 3.11]{Sato-Takagi}. This implies that $X$ is smooth.
\end{rem}
\begin{proof}
   (1) and (2) follows from \cite[Theorem 5.1]{Tot19} and \cite[Theorem 5.2]{Baudin-Bernasconi-Kawakami}, respectively.
   Since $\rho$ is finite, 
   $\rho^{*}\Omega^{[1]}_X$ satisfies $S_2$, and thus it coincides with $\Omega^1_{\mathbb{G}^3_{m}}$.
   Then we have
   \[
   (\sigma_{*}\Omega^1_{\mathbb{G}^3_{m}})^{\Z/2\Z}=(\sigma_{*}\sigma^{*}\Omega^{[1]}_X)^{\Z/2\Z}=\Omega_X^{[1]}.
   \]
   We have $\Omega^{1}_{\mathbb{G}^3_m}=\sO_{\mathbb{G}^3_m}\frac{dx}{x}\oplus \sO_{\mathbb{G}^3_m}\frac{dy}{y}\oplus \sO_{\mathbb{G}^3_m}\frac{dz}{z}$, and each of $\{\frac{dx}{x},\frac{dy}{y},\frac{dz}{z}\}$ is invariant under the $\Z/2\Z$-action. Thus, taking $(\sigma_{*}(-))^{\Z/2\Z}$,
   we obtain
   \[
   \Omega^{[1]}_{X}=(\sigma_{*}\Omega^1_{\mathbb{G}^3_{m}})^{\Z/2\Z}=\sO_{X}\frac{dx}{x}\oplus \sO_{X}\frac{dy}{y}\oplus \sO_{X}\frac{dz}{z},
   \]
   and (3) holds.
   Since each of $\frac{dx}{x},\frac{dy}{y},\frac{dz}{z}\in Z\Omega^{1}_{\mathbb{G}^3_m}$ descents to $Z\Omega^{[1]}_X$,
   the reflexive Cartier operator $Z\Omega^{[1]}_X\to \Omega^{[1]}_X$ is surjective. This shows (4). 
   Moreover, (5) follows from (4) by \cite[Theorem A]{Kaw4}.
   Finally, since $\Omega^{[1]}_X$ is locally free and \[
   C\colon Z\Omega^{[1]}_X\to \Omega^{[1]}_X
   \]
   is surjective, this map splits.
   Thus, we conclude that $X$ is Frobenius liftable by \cite[Theorem 3.3]{Kawakami-Takamatsu}, concluding (6).
\end{proof}

\subsection{An additional vanishing result}
In this section, we prove the vanishing $R^{d-1}\pi_{*}\sO_Y(-E)=0$ for a $d$-dimensional normal variety $X$ of lc type and a resolution $\pi\colon Y\to X$ with reduced exceptional divisor $E$.
The argument closely follows that of \cite[Theorem 3.2]{Baudin-Kawakami-Rysler}, where the vanishing 
$R^{d-1}\pi_{*}\sO_Y=0$ is proven under the assumption that $X$ is of klt type. For the reader’s convenience, we include the proof here.

\begin{thm}\label{thm:lc}
    Let $X$ be a $d$-dimensional Noetherian excellent normal irreducible scheme that admits a dualizing complex.
    Suppose that $X$ is of lc type, i.e., there exists an effective $\Q$-divisor $B$ such that $(X,B)$ is lc.
    Let $\pi\colon Y\to X$ be a resolution with reduced exceptional divisor $E$.
    Then 
    \[
    \dim\Supp R^j\pi_{*}\sO_Y(-E)\leq d-j-2.
    \]
    In particular, 
    \[
    R^{d-1}\pi_{*}\sO_Y(-E)=0.
    \]
\end{thm}
\begin{proof}
    By induction on the dimension and by localizing, it suffices to show that \[
    R^{d - 1}\pi_*\sO_Y(-E) = 0.\]
    By \cite[Theorem II.7.17]{Har}, there exists a closed subscheme $Z\subseteq X$ such that $\pi$ is the blow-up of $X$ along $Z$. Then $H\coloneqq\pi^{-1}(Z)$ is $\pi$-anti-ample (see \cite[\href{https://stacks.math.columbia.edu/tag/02NS}{Tag 02NS}]{stacks-project} and \cite[\href{https://stacks.math.columbia.edu/tag/02OS}{Tag 02OS}]{stacks-project}). 
    We take $n\gg 0$ such that $R^{d-1}\pi_*\sO_Y(-nH-E)=0$ by relative Serre vanishing. We write
    \begin{align*}
        nH=\sum_{i\in I}r_iE_i+G,
    \end{align*}
    for some positive integers $r_i>0$,
    where the $E_i$'s are $\pi$-exceptional components of $H$. Consider the short exact sequence
    \begin{align*}
        0\to\sO_Y(-nH-E)\to\sO_Y(-\sum_{i\in I}r_iE_i-E)\to\sO_G(-\sum_{i\in I}r_iE_i-E)\to 0.
    \end{align*}
    Since all fibers of $G \to \pi(G)$ have dimension at most $d - 2$, we have
    \[
    R^{d-1}\pi_{*}\sO_G(-\sum_{i\in I}r_iE_i-E)=0.
    \]
    Thus we obtain
    \[
    R^{d-1}\pi_{*}\sO_Y(-\sum_{i\in I}r_iE_i-E)=0.
    \]
    
    To conclude the proof, we are then left to show the following:
    \begin{cl}
        If $R^{d-1}\pi_{*}\sO_Y(-\sum_{i\in I} n_iE_i-E)=0$ for some $(n_i)_{i\in I}\in\Z_{\geq 0}$ satisfying $\sum_{i\in I} n_i\geq 1$, then there exists $j\in I$ such that $n_j\geq 1$ and \[R^{d-1}\pi_{*}\sO_Y\left(-(n_j-1)E_j-\sum_{i\in I\setminus\{j\}} n_iE_i-E\right)=0.\] 
    \end{cl}
    \noindent \textit{Proof of the claim.}
    Since $X$ is of lc type, there exists an effective $\Q$-divisor $B$ such that
    \[
    K_Y+\sum_{i\in I}a_iE_i+E+\pi_{*}^{-1}B=\pi^{*}(K_X+B)
    \]
    for some $a_i\in\Q_{\leq 0}$. 
    Let $J\coloneqq \{i\in I \mid n_i-a_i> 0\}\subseteq I$.
    Note that $J\neq \emptyset$ since $\sum_{i\in I} n_i\geq 1$ and $a_i\leq 0$. Let
    \begin{align*}
        t\coloneqq\underset{i\in J}{\max}\left\{\frac{n_i-a_i}{r_i}\right\}\in\Q_{>0},
    \end{align*}
    and let $j\in J$ be an index where the maximum is attained. We then have
    \[E_j\not\subseteq \Supp\left(t\left(\sum_{i\in I} r_iE_i\right)-\sum_{i\in J} (n_i-a_i )E_i\right),\]
    and thus
    \[
    -\left.\left(\sum_{i\in J} (n_i-a_i )E_i)\right)\right|_{E_j}=\left.\left(-tG-t\sum_{i\in I} r_iE_i\right)\right|_{E_j}+\left.\left(tG+t\sum_{i\in I} r_iE_i-\sum_{i\in J} (n_i-a_i )E_i\right)\right|_{E_j}
    \]
    is $\pi|_{E_j}$-big since $-G-\sum_{i\in I} r_iE_i$ is $\pi$-ample.
    Therefore,
    \begin{align*}
        -\left.\left(K_{Y}+\sum_{i\in I} n_iE_i+E\right)\right|_{E_j} & \sim_{\Q,\pi|_{E_j}} -\left.\left(\sum_{i\in I} (n_i-a_i)E_i\right)\right|_{E_j}+\pi^{-1}_{*}B|_{E_j}\\
        &= -\left.\left(\sum_{i\in J} (n_i-a_i )E_i\right)\right|_{E_j}+\left.\left(\sum_{i\in I\setminus J} (a_i-n_i )E_i\right)\right|_{E_j}+\pi^{-1}_{*}B|_{E_j}
    \end{align*}
    is also $\pi|_{E_j}$-big.

    Consider the short exact sequence
    \begin{multline*}
       0\to \sO_Y\left(-\sum_{i\in I} n_iE_i-E\right)\to  \sO_Y\left(-(n_j-1)E_j-\sum_{i\in I\setminus\{j\}} n_iE_i-E\right)\\\to \sO_{E_j}\left(E_j-\sum_{i\in I} n_iE_i-E\right)\to 0. 
    \end{multline*}
     Since we aim to show that \[R^{d-1}\pi_{*}\sO_Y\left(-(n_j-1)E_j-\sum_{i\in I\setminus\{j\}} n_iE_i-E\right)=0,\] it suffices to show that 
     \[
     R^{d-1}\pi_{*}\sO_{E_j}\left(E_j-\sum_{i\in I} n_iE_i-E\right)=0.
     \]
    If $\dim \pi(E_j)>0$, then this is immediate since then fibers of $E_j \to \pi(E_j)$ have dimension $\leq d - 2$. If $\dim \pi(E_j)=0$, then
    \begin{align*}
       R^{d-1}\pi_{*}\sO_{E_j}\left(E_j-\sum_{i\in I} n_iE_i-E\right)&=H^{d-1}\left(E_j, \sO_{E_j}\left(E_j-\sum_{i\in I} n_iE_i-E\right)\right)\\
    &\cong H^0\left(E_j, \sO_{E_j}\left(K_{E_j}-E_j+\sum_{i\in I} n_iE_i+E\right)\right)^{\vee} \\
    &\cong H^0\left(E_j, \sO_{E_j}\left(K_{Y}+\sum_{i\in I} n_iE_i+E\right)\right)^{\vee}.
    \end{align*}
    Since $-\left(K_{Y}+\sum_{i\in I} n_iE_i+E\right)$ is $\pi|_{E_j}$-big, we conclude.

\end{proof}


\section*{Acknowledgements}
The author was supported by JSPS KAKENHI Grant Number JP24K16897 and by the Inamori Foundation.

\bibliography{hoge.bib}
\bibliographystyle{alpha}

\bigskip

\end{document}